\documentclass[a4paper,12pt]{amsart}
\usepackage{amsmath}
\usepackage{amsthm}
\usepackage{amsfonts}
\usepackage{amssymb,array}
\usepackage{enumerate}
\usepackage[dvips]{graphicx}
\usepackage{comment}
\usepackage{ccaption}
\usepackage[all]{xy}

\theoremstyle{definition}
\newtheorem{define}{Definition}[section]
\theoremstyle{plain}

\theoremstyle{plain}
\newtheorem{theorem}[define]{Theorem}
\theoremstyle{plain}
\newtheorem{lemma}[define]{Lemma}
\theoremstyle{plain}
\newtheorem{proposition}[define]{Proposition}
\theoremstyle{plain}
\newtheorem{corollary}[define]{Corollary}
\theoremstyle{remark}

\theoremstyle{definition}

\theoremstyle{definition}

\theoremstyle{definition}

\theoremstyle{definition}

\theoremstyle{definition}

\theoremstyle{definition}

\theoremstyle{definition}

\theoremstyle{definition}

\theoremstyle{definition}

\numberwithin{equation}{section}
\numberwithin{figure}{section}
\numberwithin{table}{section}

\title[Decomposition for rank-one reductive spherical pairs]
{A Cartan decomposition for non-symmetric reductive spherical pairs of rank-one type 
and its application to visible actions}
\author[A. Sasaki]{Atsumu SASAKI}

\thanks{
	This was partially supported by 
	Grand-in-Aid for Scientific Research (C) (No. 15K04797), 
	Japan Society for the Promotion of Science. 
}


\address[A. Sasaki]{Department of Mathematics, Faculty of Science, Tokai University, 
4-1-1, Kitakaname, Hiratsuka, Kanagawa, 259-1292, Japan. }
\email{atsumu@tokai-u.jp}

\begin{document}

\begin{abstract}
A Cartan decomposition for symmetric pairs plays an important role 
to study not only orbit geometry of the symmetric spaces but also harmonic analysis on them. 
For non-symmetric reductive pairs, 
there are examples of generalizations of Cartan decompositions 
for some spherical complex homogeneous spaces such as 
complex line bundles over the complexified Hermitian symmetric spaces 
and triple spaces. 
This paper provides new examples of a Cartan decomposition for non-symmetric reductive pairs, 
namely, reductive non-symmetric spherical pairs of rank-one type. 
We also show that the action of some compact group on a non-symmetric reductive spherical 
homogeneous space of rank-one type is strongly visible. 
\end{abstract}

\maketitle



\section{Introduction}
\label{sec:intro}

Let $G$ be a connected real semisimple Lie group 
and $H$ a connected closed subgroup of $G$ which is reductive in $G$. 
We take a Cartan involution $\theta $ of $G$ satisfying $\theta (H)=H$ 
and set $K:=G^{\theta }$. 
Then, $K$ is a maximal compact subgroup of $G$. 
Let $\mathfrak{g}_0$, $\mathfrak{k}_0$ and $\mathfrak{h}_0$ 
be the Lie algebras of $G$, $K$ and $H$, respectively. 
We write $\mathfrak{g}_0=\mathfrak{k}_0+\mathfrak{p}_0$ 
for the corresponding Cartan decomposition of the Lie algebra $\mathfrak{g}_0$ 
and $\mathfrak{g}_0=\mathfrak{h}_0+\mathfrak{q}_0$ for the direct sum decomposition 
with respect to the Killing form of $\mathfrak{g}_0$. 
Then, $\mathfrak{g}_0$ is decomposed into the direct sum as follows: 
\begin{align*}
\mathfrak{g}_0=\mathfrak{k}_0\cap \mathfrak{h}_0+\mathfrak{k}_0\cap \mathfrak{q}_0
	+\mathfrak{p}_0\cap \mathfrak{h}_0+\mathfrak{p}_0\cap \mathfrak{q}_0. 
\end{align*}
In this setting, 
if there exists an abelian subspace $\mathfrak{a}_0$ in $\mathfrak{p}_0\cap \mathfrak{q}_0$ such that 
\begin{align}
\label{eq:cartan}
G=K(\exp \mathfrak{a}_0)H, 
\end{align}
then we say that (\ref{eq:cartan}) is a Cartan decomposition for the pair $(G,H)$, 
or for the homogeneous space $G/H$. 

If $H$ coincides with $K$, 
then (\ref{eq:cartan}) is a usual Cartan decomposition. 
In the case where $(G,H)$ is a symmetric pair (note that $H$ is not necessary compact), 
the decomposition (\ref{eq:cartan}) holds (cf. Flensted-Jensen \cite{fj} and Rossmann \cite{rossmann}). 
Recently, 
we have investigated a decomposition $G_{\mathbb{C}}=G_uBH_{\mathbb{C}}$ 
with an abelian $B$ in $G_{\mathbb{C}}$ 
for some complex homogeneous spaces $G_{\mathbb{C}}/H_{\mathbb{C}}$ which are not symmetric but spherical 
such as $SO(4n+2,\mathbb{C})/SL(2n+1,\mathbb{C})$, $E_6(\mathbb{C})/Spin(10,\mathbb{C})$ 
(\cite{dedicata,pja}), triple space $(G'\times G'\times G')/G'$ where $G'=SL(2,\mathbb{C})$ 
(\cite{krotz}), 
and with a subset $B$ such as $SL(2n+1,\mathbb{C})/Sp(n,\mathbb{C})$ (\cite{jms}) 
and $SO(8,\mathbb{C})/G_2(\mathbb{C})$ (\cite{apam}). 
We expect that 
non-symmetric reductive real spherical pairs $(G,H)$, 
which are in connection with finite-multiplicity property of representations 
(see \cite{kobayashi-oshima} and references therein), 
have a Cartan decomposition (\ref{eq:cartan}) for some abelian $\mathfrak{a}_0$. 

From this viewpoint, 
the author is studying a Cartan decomposition (\ref{eq:cartan}) for reductive spherical pairs. 
This class contains complex symmetric pairs. 
Then, our interest is non-symmetric spherical ones. 

Here is a quick review on reductive spherical pairs. 
Let $G_{\mathbb{C}}$ be a connected complex semisimple Lie group 
and $H_{\mathbb{C}}$ a complex closed subgroup of $G_{\mathbb{C}}$. 
The pair $(G_{\mathbb{C}},H_{\mathbb{C}})$ is called {\it spherical}, 
or the complex homogeneous space $G_{\mathbb{C}}/H_{\mathbb{C}}$ is {\it spherical}, 
if a Borel subgroup of $G_{\mathbb{C}}$ has an open orbit in $G_{\mathbb{C}}/H_{\mathbb{C}}$. 
The classification of reductive spherical pairs $(G_{\mathbb{C}},H_{\mathbb{C}})$, 
namely, it is spherical and $H_{\mathbb{C}}$ is a reductive Lie group, has been given by 
Kr\"amer \cite{kramer} when $G_{\mathbb{C}}$ is simple 
and by Brion \cite{brion}, Mykytyuk \cite{mykytyuk} when $G_{\mathbb{C}}$ is semisimple. 

In this paper, 
we consider non-symmetric reductive spherical pairs $(G_{\mathbb{C}},H_{\mathbb{C}})$ 
which are of {\it rank-one type}, 
namely, $(G_{\mathbb{C}},H_{\mathbb{C}})$ satisfies one of Table \ref{table:rank-one}. 

\begin{table}[htbp]
\begin{align*}
\begin{array}{lccc}
\hline 
\mbox{Type} & G_{\mathbb{C}} & H_{\mathbb{C}} \\
\hline 
\mbox{R-1} & SO(7,\mathbb{C}) & G_2(\mathbb{C}) \\
\mbox{R-1}' & Spin(7,\mathbb{C}) & G_2(\mathbb{C}) \\
\mbox{R-2} & G_2(\mathbb{C}) & SL(3,\mathbb{C}) \\
\hline 
\end{array}
\end{align*}
\caption{Non-symmetric reductive spherical pairs of rank-one type}
\label{table:rank-one}
\end{table}

Now, we remark on Type R-1 and Type R-1$'$. 
The double covering group homomorphism 
$Spin(7,\mathbb{C})\to SO(7,\mathbb{C})$ induces the double covering map from 
$Spin(7,\mathbb{C})/G_2(\mathbb{C})$ (Type R-1$'$) to $SO(7,\mathbb{C})/G_2(\mathbb{C})$ (Type R-1). 
Since the definition of spherical homogeneous space is a local condition, 
it is sufficient for the classification to consider spherical homogeneous spaces up to coverings. 
In the study of (\ref{eq:cartan}), 
our argument for Type R-1$'$ is parallel to that for Type R-2, 
and our main result for Type R-1$'$ provides that for Type R-1. 
For this reason, 
we will also treat Type R-1$'$ in this paper. 

The terminology `rank-one type' comes from the geometric sense, 
namely, the corresponding homogeneous space is of rank one. 
Moreover, it can be explained from the representation theory. 
By Vinberg--Kimelfeld \cite{vinberg-kimelfeld}, 
a reductive pair $(G_{\mathbb{C}},H_{\mathbb{C}})$ is spherical 
if the space $\mathbb{C}[G_{\mathbb{C}}/H_{\mathbb{C}}]$ of regular functions on 
the complex homogeneous space $G_{\mathbb{C}}/H_{\mathbb{C}}$ is multiplicity-free 
as a representation of $G_{\mathbb{C}}$, and vice versa. 
The set of highest weights of $G_{\mathbb{C}}$ occurring 
in the multiplicity-free irreducible decomposition 
of $\mathbb{C}[G_{\mathbb{C}}/H_{\mathbb{C}}]$, called the support of $G_{\mathbb{C}}/H_{\mathbb{C}}$, 
is a semigroup. 
If $(G_{\mathbb{C}},H_{\mathbb{C}})$ is one of Table \ref{table:rank-one}, 
then the rank of the support equals one (see \cite{kramer}). 

Now, let us explain our main result. 
Let $(G_{\mathbb{C}},H_{\mathbb{C}})$ be a reductive spherical pair of rank-one type 
and $\theta $ a Cartan involution of $G_{\mathbb{C}}$ 
satisfying $\theta (H_{\mathbb{C}})=H_{\mathbb{C}}$. 
Let $\mathfrak{g},\mathfrak{h}$ be the Lie algebras of $G_{\mathbb{C}},H_{\mathbb{C}}$, respectively. 
We use the same letter $\theta$ to denote the differential automorphism on $\mathfrak{g}$. 
We set $\mathfrak{g}_u=\mathfrak{g}^{\theta }$ and $G_u:=G_{\mathbb{C}}^{\theta }$. 
Then, $\mathfrak{g}_u$ is a compact real form of $\mathfrak{g}$ and the Lie algebra of $G_u$. 
The corresponding Cartan decomposition of the Lie algebra $\mathfrak{g}$ is given by 
$\mathfrak{g}=\mathfrak{g}_u+\sqrt{-1}\mathfrak{g}_u$. 
Let $\mathfrak{q}$ be the orthogonal complement of $\mathfrak{h}$ in $\mathfrak{g}$ 
with respect to the Killing form of $\mathfrak{g}$. 
Under the setting, we give new examples of a Cartan decomposition (\ref{eq:cartan}). 
Namely, we prove: 

\begin{theorem}
\label{thm:main}
Let $(G_{\mathbb{C}},H_{\mathbb{C}})$ be a non-symmetric reductive spherical pair of rank-one type. 
Then, one can take a one-dimensional abelian subspace $\mathfrak{a}_0$ 
in $\sqrt{-1}\mathfrak{g}_u\cap \mathfrak{q}$ such that 
\begin{align*}
G_{\mathbb{C}}=G_u(\exp \mathfrak{a}_0)H_{\mathbb{C}}. 
\end{align*}
\end{theorem}

The purpose of this paper is to provide an explicit description of $\mathfrak{a}_0$ 
satisfying Theorem \ref{thm:main} under the realization of $G_{\mathbb{C}}$ and $H_{\mathbb{C}}$ 
as matrix groups. 

Thanks to Theorem \ref{thm:main}, 
we find strongly visible actions on non-symmetric complex reductive homogeneous spaces. 
We state: 

\begin{theorem}[see Theorem \ref{thm:visible-slice}]
\label{thm:visible}
Let $(G_{\mathbb{C}},H_{\mathbb{C}})$ be a non-symmetric reductive spherical pair of rank-one type. 
Then, the $G_u$-action on the complex homogeneous space $G_{\mathbb{C}}/H_{\mathbb{C}}$ 
is strongly visible (see Section \ref{sec:visible} for the definition). 
\end{theorem}

This paper is organized as follows. 
In Section \ref{sec:preliminaries}, 
we explain a matrix realization of Lie groups which appear in Table \ref{table:rank-one}. 
In Sections \ref{sec:g2sl3}--\ref{sec:so7g2}, 
we prove Theorem \ref{thm:main} for each reductive spherical pair of rank-one type 
by giving an explicit description of an abelian group $A$. 
In particular, we deal with Type R-2 in Section \ref{sec:g2sl3}, 
Type R-1$'$ in Section \ref{sec:spin7g2} and Type R-1 in Section \ref{sec:so7g2}. 
In Section \ref{sec:visible}, 
we show Theorem \ref{thm:visible}, in particular, 
we explain that Theorem \ref{thm:main} is an essential part of our proof of Theorem \ref{thm:visible}. 

The author would like to express his thank to an anonymous referee 
for careful comments and suggestions. 


\section{Preliminaries}
\label{sec:preliminaries}

\subsection{Realization of exceptional Lie group of Type G$_2$}
\label{subsec:realization-g2}

In this subsection, we explain realizations of $G_2(\mathbb{C})$, 
and its maximal compact subgroup $G_2$ as subgroups of 
the special orthogonal group on the complexified Cayley algebra. 

Let $\mathfrak{C}$ be a Cayley algebra over $\mathbb{R}$. 
This algebra $\mathfrak{C}$ has the standard basis $\{ e_0,\ldots ,e_7\} $ 
with $e_0$ as the unit element of $\mathfrak{C}$ such that the following relations hold, 
which we will fix in this paper: 
\begin{align}
\label{eq:relation}
\left\{ 
\begin{array}{l}
e_0^2=e_0,~e_i^2=-e_0~(i=1,2,\ldots ,7),~e_1e_2=-e_2e_1=e_3,\\
e_1e_4=-e_4e_1=e_5,~e_1e_6=-e_6e_1=e_7,~e_2e_5=-e_5e_2=e_7,\\
e_2e_6=-e_6e_2=e_4,~e_3e_5=-e_5e_3=e_6,~e_3e_4=-e_4e_3=e_7. 
\end{array}
\right. 
\end{align}
Let $(\cdot ,\cdot )$ be an inner product on $\mathfrak{C}$ 
satisfying $(e_i,e_j)=\delta _{ij}$ ($0\leq i,j\leq 7$). 
We denote by $\operatorname{Re}(\mathfrak{C})=\mathbb{R}e_0$ the real part of $\mathfrak{C}$ 
by $\operatorname{Im}(\mathfrak{C})=\mathbb{R}e_1+\cdots +\mathbb{R}e_7$ 
the imaginary part of $\mathfrak{C}$. 
Then, $\operatorname{Im}(\mathfrak{C})$ is the orthogonal complement of $\operatorname{Re}(\mathfrak{C})$ 
in $\mathfrak{C}$ with respect to $(\cdot ,\cdot )$. 

We identify the special orthogonal group $SO(\mathfrak{C},(\cdot ,\cdot ))$ 
with $SO(8)=\{ g\in SL(8,\mathbb{R}):{}^tgg=I_8\} $ 
where ${}^tg$ denotes the transposed matrix of $g$ and $I_8$ the unit matrix. 
Similarly, we have 
$SO(\operatorname{Im}(\mathfrak{C}),(\cdot ,\cdot ))
=\{ g\in SO(\mathfrak{C},(\cdot ,\cdot )):ge_0=e_0\} \simeq SO(7)$. 
We define $G_2$ as the automorphism group 
$\operatorname{Aut}_{\mathbb{R}}(\mathfrak{C})$ of $\mathfrak{C}$, namely, 
\begin{align*}
G_2=\{ g\in SO(8):(gx)(gy)=g(xy)~(x,y\in \mathfrak{C})\} . 
\end{align*}
Then, $G_2$ is a connected and simply connected compact simple Lie group of exceptional type 
with Lie algebra $\mathfrak{g}_2$. 
By definition, 
any element $g\in G_2$ satisfies $ge_0=e_0$, 
from which $G_2$ is a subgroup of $SO(7)$. 
In particular, $\operatorname{Im}(\mathfrak{C})$ is $G_2$-invariant. 

Let $\mathfrak{C}_{\mathbb{C}}=\mathfrak{C}\otimes _{\mathbb{R}}\mathbb{C}$ be 
the complexified Cayley algebra and 
$\operatorname{Im}(\mathfrak{C}_{\mathbb{C}})$ the complexification of 
$\operatorname{Im}(\mathfrak{C})$. 
We extend the symmetric bilinear form $(\cdot ,\cdot )$ on $\mathfrak{C}$ 
to a complex symmetric bilinear form 
$(\cdot ,\cdot ):\mathfrak{C}_{\mathbb{C}}\times \mathfrak{C}_{\mathbb{C}}\to \mathbb{C}$ 
on $\mathfrak{C}_{\mathbb{C}}$, namely, 
$(v,w)=v_0w_0+\cdots +v_7w_7$ for $v=v_0e_0+\cdots +v_7e_7$, $w=w_0e_0+\cdots +w_7e_7
\in \mathfrak{C}_{\mathbb{C}}$. 
We identify the complex special orthogonal group $SO(\mathfrak{C}_{\mathbb{C}},(\cdot ,\cdot ))$ 
with $SO(8,\mathbb{C})=\{ g\in SL(8,\mathbb{C}):{}^tgg=I_8\} $ 
and $SO(\operatorname{Im}(\mathfrak{C}_{\mathbb{C}}),(\cdot ,\cdot ))$ with $SO(7,\mathbb{C})$ 
with respect to the $\mathbb{C}$-basis $\{ e_0,\ldots ,e_7\} $. 
We set $G_2(\mathbb{C}):=\operatorname{Aut}_{\mathbb{C}}(\mathfrak{C}_{\mathbb{C}})$, namely, 
\begin{align*}
G_2(\mathbb{C})
=\{ g\in SO(8,\mathbb{C}):(gx)(gy)=g(xy)~(x,y\in \mathfrak{C}_{\mathbb{C}})\} . 
\end{align*}
Then, $G_2(\mathbb{C})$ is a connected and simply connected complex simple Lie group 
of exceptional type with Lie algebra 
$\mathfrak{g}_2(\mathbb{C}):=\mathfrak{g}_2\otimes _{\mathbb{R}}\mathbb{C}$. 
This is a subgroup of $SO(7,\mathbb{C})$, and then, 
$\operatorname{Im}(\mathfrak{C}_{\mathbb{C}})$ is $G_2(\mathbb{C})$-invariant. 

The subgroup 
$G'(\mathbb{C}):=\{ g\in G_2(\mathbb{C}):ge_1=e_1\} $ of $G_2(\mathbb{C})$ 
is isomorphic to the special linear group $SL(3,\mathbb{C})$, 
and $G'(\mathbb{C})\cap G_2$ to the special unitary group $SU(3)$. 
In this paper, 
we shall identify $G'(\mathbb{C})$ and $G'(\mathbb{C})\cap G_2$ with $SL(3,\mathbb{C})$ and $SU(3)$, 
respectively: 
\begin{align*}
SL(3,\mathbb{C})&=\{ g\in G_2(\mathbb{C}):ge_1=e_1\} ,~
SU(3)=\{ g\in G_2:ge_1=e_1\} . 
\end{align*}

\subsection{Realization of spinor group $Spin(7,\mathbb{C})$}
\label{subsec:realization-spin7}

Next, we realize the spinor group $Spin(7,\mathbb{C})$ 
as a subgroup of $SO(8,\mathbb{C})=SO(\mathfrak{C}_{\mathbb{C}},(\cdot ,\cdot ))$ as follows. 

We define a subgroup $B_3(\mathbb{C})$ of $SO(8,\mathbb{C})$ by 
\begin{multline*}
B_3(\mathbb{C}):=\{ g\in SO(8,\mathbb{C}):\mbox{there exists $g_0\in SO(7,\mathbb{C})$} \\
\mbox{such that $(g_0x)(gy)=g(xy)$ ($x,y\in \mathfrak{C}_{\mathbb{C}}$)}\} 
\end{multline*}
and $B_3$ of $SO(8)$ by 
\begin{multline*}
B_3:=\{ g\in SO(8):\mbox{there exists $g_0\in SO(7)$} \\
\mbox{such that $(g_0x)(gy)=g(xy)$ ($x,y\in \mathfrak{C}$)}\} . 
\end{multline*}
Then, $G_2(\mathbb{C})$ and $G_2$ are subgroups of $B_3(\mathbb{C})$ and $B_3$, 
respectively. 
In particular, they are of the forms: 
\begin{align*}
G_2(\mathbb{C})=\{ g\in B_3(\mathbb{C}):ge_0=e_0\} ,~
G_2=\{ g\in B_3:ge_0=e_0\} .
\end{align*}

We will see that $B_3(\mathbb{C})$ is isomorphic to $Spin(7,\mathbb{C})$ as follows. 
Let $g$ be an element of $B_3(\mathbb{C})$. 
By definition, we take $g_0\in SO(7,\mathbb{C})$ such that 
$(g_0x)(gy)=g(xy)$ ($x,y\in \mathfrak{C}_{\mathbb{C}}$). 
By the principle of triality in $SO(8,\mathbb{C})$, 
the existence of such $g_0$ is unique. 
Under the notation, this yields the following map 
\begin{align}
\label{eq:covering}
\pi :B_3(\mathbb{C})\to SO(7,\mathbb{C}),\quad g\mapsto \pi (g)=g_0
\end{align}
Then, the following equality holds for any $g_1,g_2\in B_3(\mathbb{C})$: 
\begin{align*}
(g_1g_2)(xy)&=g_1(g_2(xy))\\
&=g_1((\pi (g_2)x)(g_2y))\\
&=(\pi (g_1)(\pi (g_2)x))(g_1(g_2y))\\
&=((\pi (g_1)\pi (g_2))x)((g_1g_2)y) \quad (x,y\in \mathfrak{C}_{\mathbb{C}}). 
\end{align*}
This means that $\pi (g_1g_2)=\pi (g_1)\pi (g_2)$ for any $g_1,g_2\in B_3(\mathbb{C})$, from which 
$\pi$ is a group homomorphism. 
On the other hand, 
let us take an element $g_0\in SO(7,\mathbb{C})$. 
It follows from the principle of triality to $g_0$ that 
there exists $g\in SO(8,\mathbb{C})$ such that $(g_0x)(gy)=g(xy)$ ($x,y\in \mathfrak{C}_{\mathbb{C}})$, 
from which $g\in B_3(\mathbb{C})$. 
This means that $\pi$ is a surjective map. 
Moreover, if $g,g'\in B_3(\mathbb{C})$ satisfy $\pi (g)=\pi (g')$, 
then $g'$ coincides either $g$ or $-g$. 
Thus, $\pi$ is a double covering map. 
Therefore, $B_3(\mathbb{C})$ is isomorphic to $Spin(7,\mathbb{C})$. 
Similarly, we have $B_3\simeq Spin(7)$. 

Throughout this paper, 
$Spin(7,\mathbb{C})$ means $B_3(\mathbb{C})$ and $Spin(7)$ means $B_3$. 

\subsection{Cartan involution}
\label{subsec:cartan involution}

As mentioned before, 
we take a matrix realization of each complex simple Lie group and its maximal compact subgroup. 
Then, there exists a Cartan involution of each complex Lie group such that 
the fixed point set equals the maximal compact subgroup. 
In this subsection, we express such a Cartan involution as follows. 

Let us define an involutive automorphism $\theta $ on $SO(8,\mathbb{C})$ by 
\begin{align}
\label{eq:theta}
\theta (g)=\overline{g}\quad (g\in SO(8,\mathbb{C})). 
\end{align} 
Here, $\overline{g}$ stands for the complex conjugation $\overline{g}$ of 
$g\in SO(8,\mathbb{C})$. 
Then, $\theta $ is a Cartan involution of $SO(8,\mathbb{C})$ 
and its fixed point set $SO(8)$ is a maximal compact subgroup of $SO(8,\mathbb{C})$. 
The restrictions of $\theta$ to $B_3(\mathbb{C})$ and $G_2(\mathbb{C})$ becomes 
Cartan involutions on $B_3(\mathbb{C})$ and $G_2(\mathbb{C})$, respectively. 
Thus, $B_3$ is a maximal compact subgroup of $B_3(\mathbb{C})$ and 
$G_2$ is that of $G_2(\mathbb{C})$. 

\subsection{Unit sphere and complex unit sphere}
\label{subsec:sphere}

In this subsection, 
we review the facts on transitive actions and isotropy subgroups. 

Let $W=(\mathbb{R}e_1)^{\bot}=\mathbb{R}e_2+\cdots +\mathbb{R}e_7$ 
be the orthogonal complement of $\mathbb{R}e_1$ 
in $\operatorname{Im}(\mathfrak{C})$. 
For a vector space $V$ $(=\mathfrak{C},\operatorname{Im}(\mathfrak{C}),W)$ over $\mathbb{R}$, 
we write $S(V)=\{ v\in V:(v,v)=1\} $ for the unit sphere in $V$ with respect to $(\cdot ,\cdot )$. 
As $G_2\subset SO(7)$ and $Spin(7)\subset SO(8)$, 
$Spin(7)$ acts on $S(\mathfrak{C})=S^7$ and $G_2$ acts on $S(\operatorname{Im}(\mathfrak{C}))=S^6$. 
Further, 
It is known that 

\begin{lemma}[{\cite{borel,m-s}}]
\label{lem:sphere}
\begin{enumerate}
	\item The $Spin(7)$-action on $S^7$ is transitive, 
	and then $S^7$ is diffeomorphic to $Spin(7)/G_2$. 
	\item The $G_2$-action on $S^6$ is transitive, 
	and then $S^6$ is diffeomorphic to $G_2/SU(3)$. 
	\item The $SU(3)$-action on $S(W)=S^5$ is transitive. 
\end{enumerate}
\end{lemma}

Next, 
$V_{\mathbb{C}}$ ($=\mathfrak{C}_{\mathbb{C}},\operatorname{Im}(\mathfrak{C}_{\mathbb{C}})$) 
denotes the complexification of $V$. 
We set $S(V_{\mathbb{C}}):=\{ v\in V_{\mathbb{C}}:(v,v)=1\} $ 
as the complex unit sphere in $V_{\mathbb{C}}$. 
Then, $Spin(7,\mathbb{C})$ acts on $S(\mathfrak{C}_{\mathbb{C}})=S_{\mathbb{C}}^7$ 
and $G_2(\mathbb{C})$ acts on $S(\operatorname{Im}(\mathfrak{C}_{\mathbb{C}}))=S_{\mathbb{C}}^6$. 

\begin{lemma}
\label{lem:complex-sphere}
\begin{enumerate}
	\item The $Spin(7,\mathbb{C})$-action on $S_{\mathbb{C}}^7$ is transitive, 
	and then $S_{\mathbb{C}}^7$ is biholomorphic to $Spin(7,\mathbb{C})/G_2(\mathbb{C})$. 
	\item The $G_2(\mathbb{C})$-action on $S_{\mathbb{C}}^6$ is transitive, 
	and then 
	$S_{\mathbb{C}}^6$ is biholomorphic to $G_2(\mathbb{C})/SL(3,\mathbb{C})$. 
\end{enumerate}
\end{lemma}

\begin{proof}
The isomorphism $S_{\mathbb{C}}^7\simeq Spin(7,\mathbb{C})/G_2(\mathbb{C})$ has been proved in 
\cite[Lemma 2.2]{apam}. 
The idea of the proof is based on \cite[Proposition 2 in Section 3]{murakami}. 
Similarly, one can prove $S_{\mathbb{C}}^6\simeq G_2(\mathbb{C})/SL(3,\mathbb{C})$. 
Let us see it briefly. 

We recall that $SL(3,\mathbb{C})$ is the isotropy subgroup of $G_2(\mathbb{C})$ 
at $e_1\in S_{\mathbb{C}}^6$. 
Then, we have a natural embedding 
\begin{align*}
\iota :G_2(\mathbb{C})/SL(3,\mathbb{C})\to S_{\mathbb{C}}^6,\quad 
gSL(3,\mathbb{C})\mapsto ge_1. 
\end{align*}
In particular, 
this is an injective map. 
On the other hand, the complex dimension of $G_2(\mathbb{C})/SL(3,\mathbb{C})$ equals six, 
which coincides with the that of $S_{\mathbb{C}}^6$. 
This implies $\iota$ is surjective, 
and then we obtain $G_2(\mathbb{C})/SL(3,\mathbb{C})\simeq S_{\mathbb{C}}^6$. 
\end{proof}

We note that both $Spin(7,\mathbb{C})/G_2(\mathbb{C})$ and $G_2(\mathbb{C})/SL(3,\mathbb{C})$ 
are non-symmetric homogeneous spaces. 

\subsection{Transitive actions on unit spheres}
\label{subsec:transitive}

In this subsection, 
we consider a description of a vector space 
under the setting that a compact group acts transitively on the unit sphere. 
This subsection is based on the reference \cite{imrn-irreducible}. 

Let $V$ be a vector space over $\mathbb{R}$ equipped with an inner product $(\cdot ,\cdot )$ 
and $G$ a subgroup of the orthogonal group $O(V,(\cdot ,\cdot ))$. 
Then, we have the following lemma. 
Here, $\mathbb{R}_{\geq 0}$ denotes the set of non-negative real numbers. 

\begin{lemma}[see {\cite[Lemma 5.1]{imrn-irreducible}}]
\label{lem:transitive}
Let $G$ be a subgroup of $O(V,(\cdot ,\cdot ))$ acting transitively on the unit sphere $S(V)$. 
Then, $V$ is written as $V=G\cdot (\mathbb{R}_{\geq 0}v_0)$ for a non-zero element $v_0\in S(V)$. 
\end{lemma}

\begin{proof}
Let $v$ be a non-zero element of $V$. 
We set $r:=\sqrt{(v,v)}>0$. 
Then, $v_1:=v/r$ is an element of $S(V)$. 
Since $G$ acts transitively on $S(V)$, we write $v_1=g\cdot v_0$ for some $g\in G$. 
Thus, we obtain $v=g\cdot rv_0$, which is an element of $G\cdot (\mathbb{R}_{\geq 0}v_0)$. 
\end{proof}

Next, let $G$ act on the complexification $V_{\mathbb{C}}=V+\sqrt{-1}V$ diagonally, 
namely, $g\cdot (v_1+\sqrt{-1}v_2):=gv_1+\sqrt{-1}gv_2$ for $g\in G,v_1,v_2\in V$. 

We denote by $(\mathbb{R}v_0)^{\bot}$ the orthogonal complement of $\mathbb{R}v_0$ in $V$. 
Then, we have: 

\begin{lemma}
\label{lem:vc}
Retain the setting as in Lemma \ref{lem:transitive}. 
Suppose that the isotropy subgroup $G_{v_0}$ of $G$ at $v_0$ 
acts transitively on the unit sphere $S((\mathbb{R}v_0)^{\bot})$. 
Then, $V_{\mathbb{C}}$ is expressed as 
\begin{align*}
V_{\mathbb{C}}=G\cdot (\mathbb{R}_{\geq 0}v_0+\sqrt{-1}(\mathbb{R}v_0\oplus \mathbb{R}w_0)) 
\end{align*}
for an element $w_0\in S((\mathbb{R}v_0)^{\bot})$. 
\end{lemma}

\begin{proof}
Our proof is based on \cite[Lemma 5.2]{imrn-irreducible}. 
Let us see it briefly. 
Applying Lemma \ref{lem:transitive} to the case where $G_{v_0}$ acts transitively 
on $S((\mathbb{R}v_0)^{\bot})$, 
the vector space 
$(\mathbb{R}v_0)^{\bot}$ is written as 
$(\mathbb{R}v_0)^{\bot}=G_{v_0}\cdot \mathbb{R}_{\geq 0}w_0$ 
for an element $w_0\in S((\mathbb{R}v_0)^{\bot})$. 
On the other hand, it is clear that $G_{v_0}\cdot \mathbb{R}v_0=\mathbb{R}v_0$. 
As $V=\mathbb{R}v_0\oplus (\mathbb{R}v_0)^{\bot}$, we obtain 
\begin{align*}
V=G_{v_0}\cdot \mathbb{R}v_0\oplus G_{v_0}\cdot \mathbb{R}_{\geq 0}w_0
=G_{v_0}\cdot (\mathbb{R}v_0\oplus \mathbb{R}_{\geq 0}w_0). 
\end{align*}
Combining the above decomposition with Lemma \ref{lem:transitive}, 
we conclude 
\begin{align*}
V_{\mathbb{C}}=V+\sqrt{-1}V
&=G\cdot \mathbb{R}_{\geq 0}v_0+\sqrt{-1}(G_{v_0}\cdot (\mathbb{R}v_0\oplus \mathbb{R}_{\geq 0}w_0))\\
&=G\cdot (\mathbb{R}_{\geq 0}v_0+\sqrt{-1}(\mathbb{R}v_0\oplus \mathbb{R}_{\geq 0}w_0)). 
\end{align*}
In particular, 
since $\mathbb{R}_{\geq 0}w_0$ is contained in $\mathbb{R}w_0$, 
we obtain 
\begin{align*}
V_{\mathbb{C}}&=G\cdot (\mathbb{R}_{\geq 0}v_0+\sqrt{-1}(\mathbb{R}v_0\oplus \mathbb{R}_{\geq 0}w_0))\\
&\subset G\cdot (\mathbb{R}_{\geq 0}v_0+\sqrt{-1}(\mathbb{R}v_0\oplus \mathbb{R}w_0))
\subset V_{\mathbb{C}}. 
\end{align*}
Hence, Lemma \ref{lem:vc} has been proved. 
\end{proof}


\section{Cartan decomposition for $(G_2(\mathbb{C}),SL(3,\mathbb{C}))$}
\label{sec:g2sl3}

In this section, 
we give a proof of Theorem \ref{thm:main} for a non-symmetric reductive spherical pair 
$(G_2(\mathbb{C}),SL(3,\mathbb{C}))$ (Type R-2). 

We begin this section with the outline of our proof. 
As mentioned in Lemma \ref{lem:complex-sphere}, 
the homogeneous space $G_2(\mathbb{C})/SL(3,\mathbb{C})$ is biholomorphic to the complex unit sphere 
$S(\operatorname{Im}(\mathfrak{C}_{\mathbb{C}}))$. 
Then, we first find a real submanifold $T_1$ 
which meets every $G_2$-orbit in $S(\operatorname{Im}(\mathfrak{C}_{\mathbb{C}}))$ 
(Section \ref{subsec:g2-sphere}). 
Second, 
we give an abelian group $A_1$ such that $T_1$ is an $A_1$-orbit (Section \ref{subsec:g2-orbit}). 
After that, 
we prove Theorem \ref{thm:main} for this case (see Theorem \ref{thm:cartan-g2} for detail). 

\subsection{$G_2$-action on $S(\operatorname{Im}(\mathfrak{C}_{\mathbb{C}}))$}
\label{subsec:g2-sphere}

First, we give a decomposition of $S(\operatorname{Im}(\mathfrak{C}_{\mathbb{C}}))$ into $G_2$-orbits. 

We set 
\begin{align}
\label{eq:t1}
T_1:&=(\mathbb{R}_{\geq 0}e_1+\sqrt{-1}(\mathbb{R}e_1\oplus \mathbb{R}e_2))
	\cap S(\operatorname{Im}(\mathfrak{C}_{\mathbb{C}})). 
\end{align}

\begin{lemma}
\label{lem:g2-sphere}
The complex unit sphere $S(\operatorname{Im}(\mathfrak{C}_{\mathbb{C}}))$ is written as 
\begin{align*}
S(\operatorname{Im}(\mathfrak{C}_{\mathbb{C}}))=G_2\cdot T_1. 
\end{align*}
\end{lemma}

\begin{proof}
Retain the notation as in Section \ref{subsec:sphere}. 
We observe that $G_2$ acts transitively on $S(\operatorname{Im}(\mathfrak{C}))$ and 
the isotropy subgroup $(G_2)_{e_1}=SU(3)$ acts transitively on $S(W)$ (see Lemma \ref{lem:sphere}). 
By Lemma \ref{lem:vc}, we have 
\begin{align*}
\operatorname{Im}(\mathfrak{C}_{\mathbb{C}})
=G_2\cdot (\mathbb{R}_{\geq 0}e_1+\sqrt{-1}(\mathbb{R}e_1+\mathbb{R}e_2)). 
\end{align*}
Since $G_2$ is a subgroup of $SO(7)$, we obtain 
\begin{align*}
S(\operatorname{Im}(\mathfrak{C}_{\mathbb{C}}))
&=(G_2\cdot (\mathbb{R}_{\geq 0}e_1+\sqrt{-1}(\mathbb{R}e_1+\mathbb{R}e_2)))
	\cap S(\operatorname{Im}(\mathfrak{C}_{\mathbb{C}}))\\
&=G_2\cdot ((\mathbb{R}_{\geq 0}e_1+\sqrt{-1}(\mathbb{R}e_1+\mathbb{R}e_2)))
	\cap S(\operatorname{Im}(\mathfrak{C}_{\mathbb{C}}))\\
&=G_2\cdot T_1. 
\end{align*}
Hence, Lemma \ref{lem:g2-sphere} has been proved. 
\end{proof}

Next, we consider an explicit description of an element of $T_1$ 
in the coordinates. 
Let $v$ be an element of $T_1$. 
As $T_1\subset \mathbb{R}e_1+\sqrt{-1}(\mathbb{R}e_1+\mathbb{R}e_2)$, 
we write $v=x_1e_1+\sqrt{-1}(y_1e_1+y_2e_2)$ for some $x_1\in \mathbb{R}_{\geq 0}$ 
and some $y_1,y_2\in \mathbb{R}$. 
Then, we have 
\begin{align*}
(v,v)=(x_1+\sqrt{-1}y_1)^2+(\sqrt{-1}y_2)^2=(x_1^2-y_1^2-y_2^2)+2\sqrt{-1}x_1y_1. 
\end{align*}
Since $v\in S(\operatorname{Im}(\mathfrak{C}_{\mathbb{C}}))$, 
three real numbers $x_1,y_1,y_2$ satisfy $x_1^2-y_1^2-y_2^2=1$ and $x_1y_1=0$. 
Hence, we get $y_1=0$ and $x_1^2-y_2^2=1$. 
In particular, $x_1$ has to be a positive number. 
Therefore, $T_1$ is of the form 
\begin{align*}
T_1=\{ (\cosh \theta )e_1+\sqrt{-1}(\sinh \theta )e_2: \theta \in \mathbb{R}\} . 
\end{align*}
Here, the map $\mathbb{R}\to T_1,~\theta \mapsto (\cosh \theta )e_1+\sqrt{-1}(\sinh \theta )e_2$ 
is an embedding. 
Then, $T_1$ is a one-dimensional real submanifold in $S(\operatorname{Im}(\mathfrak{C}_{\mathbb{C}}))$. 

\subsection{$G_2$-action on $G_2(\mathbb{C})/SL(3,\mathbb{C})$}
\label{subsec:g2-orbit}

We recall from Lemma \ref{lem:complex-sphere} that 
$S(\operatorname{Im}(\mathfrak{C}_{\mathbb{C}}))$ is biholomorphic to $G_2(\mathbb{C})/SL(3,\mathbb{C})$. 
As $S(\operatorname{Im}(\mathfrak{C}_{\mathbb{C}}))=G_2\cdot T_1$, 
there exists a real submanifold $S_1$ in $G_{\mathbb{C}}/H_{\mathbb{C}}$ such that 
$T_1\simeq S_1$ 
and $G_{\mathbb{C}}/H_{\mathbb{C}}=G_u\cdot S_1$ 
via the biholomorphic diffeomorphism. 

To find $S_1$, we construct an abelian group $A_1$ as follows. 
Let us define a matrix $\delta _{(x,y)}$ by 
\begin{align*}
\delta _{(x,y)}=\left( 
	\begin{array}{cccc}
	0 & 0 & 0 & -\sqrt{-1}x \\
	0 & 0 & -\sqrt{-1}y & 0 \\
	0 & \sqrt{-1}y & 0 & 0 \\
	\sqrt{-1}x & 0 & 0 & 0 
	\end{array}
\right) . 
\end{align*}
Now, we set 
\begin{align}
\label{eq:a1-algebra}
\mathfrak{a}_1
=\{ \tau _{\theta }:=\operatorname{diag}(\delta _{(0,\theta )},\delta _{(-\theta /2,-\theta /2)}) 
:\theta \in \mathbb{R}\} 
\end{align}
and 
\begin{align}
\label{eq:a1}
A_1=\exp \mathfrak{a}_1
=\{ t_{\theta }=\operatorname{diag}(\exp \delta _{(0,\theta )}, \exp \delta _{(-\theta /2,-\theta /2)})
:\theta \in \mathbb{R}\} . 
\end{align}
Then, $A_1$ is a one-dimensional abelian group. 
We note 
\begin{align*}
\exp \delta _{(x,y)}=\left( 
	\begin{array}{cccc}
	\cosh x & 0 & 0 & -\sqrt{-1}\sinh x \\
	0 & \cosh y & -\sqrt{-1}\sinh y & 0 \\
	0 & \sqrt{-1}\sinh y & \cosh y & 0 \\
	\sqrt{-1}\sinh x & 0 & 0 & \cosh x
	\end{array}
\right) . 
\end{align*}

\begin{lemma}
\label{lem:a1-g2}
The abelian group $A_1$ is contained in $G_2(\mathbb{C})$. 
\end{lemma}

\begin{proof}[Sketch of Proof]
Let us verify that any element $t_{\theta }\in A_1$ satisfies 
\begin{align*}
(t_{\theta }e_i)(t_{\theta }e_j)=t_{\theta }(e_ie_j)\quad (0\leq i,j\leq 7)
\end{align*}
for our choice of the $\mathbb{C}$-basis $\{ e_0,\ldots ,e_7\} $ in (\ref{eq:relation}). 
In fact, the computation is straightforward from the followings: 
\begin{align*}
t_{\theta }e_1&=(\cosh \theta )e_1+\sqrt{-1}(\sinh \theta )e_2,\\
t_{\theta }e_2&=-\sqrt{-1}(\sinh \theta )e_1+(\cosh \theta )e_2,\\
t_{\theta }e_4&=(\cosh (\theta /2))e_4-\sqrt{-1}(\sinh (\theta /2))e_7,\\
t_{\theta }e_5&=(\cosh (\theta /2))e_5-\sqrt{-1}(\sinh (\theta /2))e_6,\\
t_{\theta }e_6&=\sqrt{-1}(\sinh (\theta /2))e_5+(\cosh (\theta /2))e_6,\\
t_{\theta }e_7&=\sqrt{-1}(\sinh (\theta /2))e_4+(\cosh (\theta /2))e_7 
\end{align*}
and $t_{\theta }e_i=e_i$ for $i=0,3$. 
\end{proof}

As mentioned in the proof of Lemma \ref{lem:a1-g2}, 
an element $(\cosh \theta )e_1+\sqrt{-1}(\sinh \theta )e_2\in T_1$ is written by $t_{\theta }e_1$. 
Then, the submanifold $T_1$ is expressed as 
\begin{align}
\label{eq:t1-a1}
T_1=\{ t_{\theta }e_1:\theta \in \mathbb{R}\} =A_1\cdot e_1. 
\end{align}
Hence, we set 
\begin{align}
\label{eq:s1}
S_1:=A_1SL(3,\mathbb{C})/SL(3,\mathbb{C}). 
\end{align}
Then, we have: 

\begin{lemma}
\label{lem:t1}
$S_1\simeq T_1$. 
\end{lemma}

Combining Lemma \ref{lem:g2-sphere} with Lemma \ref{lem:t1}, 
we get the decomposition of the homogeneous space 
$G_2(\mathbb{C})/SL(3,\mathbb{C})$ as follows: 

\begin{proposition}
\label{prop:g2-orbit}
$G_2(\mathbb{C})/SL(3,\mathbb{C})=G_2\cdot S_1$. 
\end{proposition}

\begin{proof}
Let $g$ be an element of $G_2(\mathbb{C})$. 
By Lemma \ref{lem:g2-sphere}, 
the element $ge_1\in S(\operatorname{Im}(\mathfrak{C}_{\mathbb{C}}))$ is written as 
$ge_1=k\cdot v_1$ for some $k\in G_u$ and $v_1\in T_1$ (see (\ref{eq:t1-a1})). 
Moreover, $v_1$ is given by $v_1=t_{\theta }e_1\in A_1\cdot e_1$ for some $t_{\theta }\in A_1$, 
from which $ge_1=(kt_{\theta })e_1$. 
This means $g^{-1}kt_{\theta }\in (G_2(\mathbb{C}))_{e_1}=SL(3,\mathbb{C})$. 
Hence, we obtain 
$gSL(3,\mathbb{C})=k\cdot t_{\theta }SL(3,\mathbb{C})\in G_2\cdot S_1$. 
\end{proof}

\subsection{Lie algebra $\mathfrak{a}_1$}
\label{subsec:a1}

In this subsection, 
we observe the Lie algebra $\mathfrak{a}_1$. 

Let $\mathfrak{g}=\mathfrak{g}_2(\mathbb{C})$, 
$\mathfrak{h}=\mathfrak{sl}(3,\mathbb{C})$ and $\mathfrak{g}_u=\mathfrak{g}_2$ 
be the Lie algebras of $G_2(\mathbb{C})$, $SL(3,\mathbb{C})$ and $G_2$, respectively. 
The differential automorphism of the Cartan involution $\theta $ of $G_2(\mathbb{C})$ 
(see (\ref{eq:theta})), 
which we use the same letter to denote, is given by 
$\theta (X)=\overline{X}$ $(X\in \mathfrak{g})$. 
Since $\overline{\delta _{(x,y)}}=\delta _{(-x,-y)}=-\delta _{(x,y)}$, 
we have $\theta (\tau _{\theta })=\tau _{(-\theta )}=-\tau _{\theta }$ 
for any $\tau _{\theta }\in \mathfrak{a}_1$. 
Hence, $\mathfrak{a}_1$ is contained in $\sqrt{-1}\mathfrak{g}_u$. 

Next, let $\mathfrak{q}$ be the orthogonal complement of $\mathfrak{h}$ in $\mathfrak{g}$ 
with respect to the Killing form on $\mathfrak{g}$. 
As $SL(3,\mathbb{C})=(G_2(\mathbb{C}))_{e_1}$, 
we write $\mathfrak{h}=\{ X\in \mathfrak{g}_2(\mathbb{C}):Xe_1=0\} $. 
Thus, $\mathfrak{a}_1$ is not contained in $\mathfrak{h}$. 
In fact, we have: 

\begin{lemma}
\label{lem:a1-q}
$\mathfrak{a}_1\subset \mathfrak{q}$. 
\end{lemma}

\begin{proof}
We will give a $\mathbb{C}$-bases of $\mathfrak{h}=\mathfrak{sl}(3,\mathbb{C})$ 
and $\mathfrak{q}$ as follows, which is based on \cite[Theorem 1.5.1]{yokota}. 

Let $E_{ij}$ $(0\leq i,j\leq 7,~i\neq j)$ 
be a $\mathbb{C}$-linear transformation on $\mathfrak{C}_{\mathbb{C}}$ 
satisfying $E_{ij}e_j=e_i$ and $E_{ij}e_k=0$ for $k\neq j$. 
We shall also use the same letter $E_{ij}$ to denote the matrix of degree eight 
corresponding to the $\mathbb{C}$-linear transformation $E_{ij}$ 
via the identification $\operatorname{End}_{\mathbb{C}}(\mathfrak{C}_{\mathbb{C}})\simeq M(8,\mathbb{C})$ 
with respect to the $\mathbb{C}$-basis $\{ e_0,\ldots ,e_7\} $. 
For $0\leq i<j\leq 7$, we set $X_{ij}:=E_{ij}-E_{ji}$. 
Then, 
\begin{multline}
\label{eq:sl3}
\{ -X_{23}+X_{45},~-X_{45}+X_{67},~X_{24}+X_{35},~-X_{25}+X_{34},\\
X_{26}+X_{37},~-X_{27}+X_{36},~X_{46}+X_{57},~-X_{47}+X_{56}\} 
\end{multline}
is a $\mathbb{C}$-basis of $\mathfrak{sl}(3,\mathbb{C})$, and 
\begin{multline}
\label{eq:q}
\{ 2X_{12}-X_{47}-X_{56},~2X_{13}-X_{46}+X_{57},~2X_{14}+X_{27}+X_{36},\\
2X_{15}+X_{26}-X_{37},~2X_{16}-X_{25}-X_{34},~2X_{17}-X_{24}+X_{35}\} 
\end{multline}
is a $\mathbb{C}$-basis of $\mathfrak{q}$. 
Since $2X_{12}-X_{47}-X_{56}\in \mathfrak{q}$ is written by $\sqrt{-1}\tau _2=2\sqrt{-1}\tau _1$, 
an arbitrary $\tau _{\theta }$ is of the form 
$\tau _{\theta }=\theta \tau _1\in \mathfrak{q}$. 
Thus, 
the Lie algebra $\mathfrak{a}_1$ is contained in $\mathfrak{q}$. 
\end{proof}

Therefore, we have shown $\mathfrak{a}_1\subset \sqrt{-1}\mathfrak{g}_u\cap \mathfrak{q}$. 

\subsection{Proof of Theorem \ref{thm:main} for $(G_2(\mathbb{C}),SL(3,\mathbb{C}))$}
\label{proof-g2}

A Cartan decomposition for the reductive spherical pair $(G_2(\mathbb{C}),SL(3,\mathbb{C}))$ is provided 
by Proposition \ref{prop:g2-orbit}. 
More precisely, we prove: 

\begin{theorem}[Theorem \ref{thm:main} for Type R-2]
\label{thm:cartan-g2}
Let $(G_{\mathbb{C}},H_{\mathbb{C}})$ be the reductive spherical pair of Type R-2, namely, 
$(G_2(\mathbb{C}),SL(3,\mathbb{C}))$. 
Then, the one-dimensional abelian group $A_1=\exp \mathfrak{a}_1$ 
with $\mathfrak{a}_1\subset \sqrt{-1}\mathfrak{g}_u\cap \mathfrak{q}$ given by (\ref{eq:a1-algebra}) 
satisfies 
\begin{align*}
G_{\mathbb{C}}=G_uA_1H_{\mathbb{C}}. 
\end{align*}
\end{theorem}

\begin{proof}
Let $g $ be an element of $G_2(\mathbb{C})$. 
By Proposition \ref{prop:g2-orbit}, 
there exists $k\in G_u=G_2$ and $t_{\theta }\in A_1$ such that 
$gH_{\mathbb{C}}=k\cdot (t_{\theta }H_{\mathbb{C}})=(kt_{\theta })H_{\mathbb{C}}$. 
Thus, we have $(kt_{\theta })^{-1}g\in H_{\mathbb{C}}$. 
We write $h:=(kt_{\theta })^{-1}g\in H_{\mathbb{C}}$. 
Then, we obtain $g=kt_{\theta }h\in G_uA_1H_{\mathbb{C}}$. 
Hence, we conclude $G_{\mathbb{C}}\subset G_uA_1H_{\mathbb{C}}$. 
Clearly, $G_{\mathbb{C}}\supset G_uA_1H_{\mathbb{C}}$.
Therefore, we get 
$G_{\mathbb{C}}=G_uA_1H_{\mathbb{C}}$. 
\end{proof}


\section{Cartan decomposition for $(Spin(7,\mathbb{C}),G_2(\mathbb{C}))$}
\label{sec:spin7g2}

In this subsection, 
we give a proof of Theorem \ref{thm:main} for 
the non-symmetric reductive spherical pair $(Spin(7,\mathbb{C}),G_2(\mathbb{C}))$. 
The proof of Theorem \ref{thm:main} for Type R-1$'$ proceeds in 
parallel with that for Type R-2 which has been discussed in Section \ref{sec:g2sl3}. 

\subsection{$Spin(7)$-action on $S(\mathfrak{C}_{\mathbb{C}})$}
\label{subsec:spin7-sphere}

In this subsection, we consider the $Spin(7)$-action 
on the complex unit sphere $S(\mathfrak{C}_{\mathbb{C}})$ (Type R-1$'$). 

The $Spin(7)$-action on $S^7=S(\mathfrak{C})$ is transitive 
and the action of the isotropy subgroup $Spin(7)_{e_0}
=G_2$ on $S^6=S(\operatorname{Im}(\mathfrak{C}))$ is also transitive (see Lemma \ref{lem:sphere}). 
It follows from Lemma \ref{lem:vc} that the complexified Cayley algebra 
$\mathfrak{C}_{\mathbb{C}}=\mathfrak{C}+\sqrt{-1}\mathfrak{C}$ is written as 
\begin{align*}
\mathfrak{C}_{\mathbb{C}}
=Spin(7)\cdot (\mathbb{R}_{\geq 0}e_0+\sqrt{-1}(\mathbb{R}e_0+\mathbb{R}e_1)). 
\end{align*}
Then, we obtain 
\begin{align*}
S(\mathfrak{C}_{\mathbb{C}})
&=Spin(7)\cdot 
(\mathbb{R}_{\geq 0}e_0+\sqrt{-1}(\mathbb{R}e_0+\mathbb{R}e_1))\cap S(\mathfrak{C}_{\mathbb{C}}). 
\end{align*}

Hence, we take a one-dimensional real submanifold $\widetilde{T}_0$ in $S(\mathfrak{C}_{\mathbb{C}})$ 
as 
\begin{align}
\label{eq:t0}
\widetilde{T}_0:
&=(\mathbb{R}_{\geq 0}e_0+\sqrt{-1}(\mathbb{R}e_0+\mathbb{R}e_1))\cap S(\mathfrak{C}_{\mathbb{C}}). 
\end{align}
Then, we have: 

\begin{lemma}
\label{lem:spin7-sphere}
The complex unit sphere $S(\mathfrak{C}_{\mathbb{C}})$ is expressed as 
\begin{align*}
S(\mathfrak{C}_{\mathbb{C}})=Spin(7)\cdot \widetilde{T}_0. 
\end{align*}
\end{lemma}

\subsection{$Spin(7)$-action on $Spin(7,\mathbb{C})/G_2(\mathbb{C})$}
\label{subsec:spin7-action}

In this subsection, we give a real submanifold 
which meets every $Spin(7)$-orbit in $Spin(7,\mathbb{C})/G_2(\mathbb{C})$. 

Let us take a matrix $d_{\theta }$ as 
\begin{align}
\label{eq:d}
d_{\theta }:=\left( 
	\begin{array}{cc}
	\cosh \theta  & -\sqrt{-1}\sinh \theta \\
	\sqrt{-1}\sinh \theta & \cosh \theta 
	\end{array}
\right) . 
\end{align}
We define a subgroup $\widetilde{A}_0$ of $SO(8,\mathbb{C})$ by 
\begin{align}
\label{eq:a0spin}
\widetilde{A}_0:=\{ \widetilde{a}_{\theta }=
	\operatorname{diag}(d_{\theta },d_{(-\theta /3)},d_{(-\theta /3)},d_{(-\theta /3)})
	:\theta \in \mathbb{R}\} 
\end{align}

\begin{lemma}
\label{lem:a0-spin}
The set $\widetilde{A}_0$ is a subgroup of $Spin(7,\mathbb{C})$. 
\end{lemma}

For the verification of Lemma \ref{lem:a0-spin}, 
we prepare the notation as follows: 
\begin{align}
\label{eq:a-theta}
a_{x}:=\operatorname{diag}(I_2,d_{x},d_{x},d_{x})\in SO(7,\mathbb{C}).
\end{align}

\begin{proof}[Sketch of Proof]
For an element $\widetilde{a}_{\theta }$ of $\widetilde{A}_0$, 
we take $a_{(2\theta /3)}\in SO(7,\mathbb{C})$. 
Then, the direct computation shows that 
\begin{align}
\label{eq:spin7}
(a_{(2\theta /3)}e_i)(\widetilde{a}_{\theta }e_j)=\widetilde{a}_{\theta }(e_ie_j) 
\quad (0\leq i,j\leq 7) 
\end{align}
for the $\mathbb{C}$-basis $\{ e_0,\ldots ,e_7\} $ of $\mathfrak{C}_{\mathbb{C}}$. 
This implies that $\widetilde{a}_{\theta }\in Spin(7,\mathbb{C})$. 
\end{proof}

By taking the same argument as for $T_1$, 
the real submanifold $\widetilde{T}_0$ is of the form 
\begin{align*}
\widetilde{T}_0=\{ (\cosh \theta )e_0+\sqrt{-1}(\sinh \theta )e_1:\theta \in \mathbb{R}\} . 
\end{align*}
Thus, we write 
\begin{align*}
\widetilde{T}_0=\widetilde{A}_0\cdot e_0. 
\end{align*}
Hence, we set 
\begin{align}
\label{eq:s0-spin7}
\widetilde{S}_0:=\widetilde{A}_0G_2(\mathbb{C})/G_2(\mathbb{C}). 
\end{align}

\begin{lemma}
\label{lem:t0}
$\widetilde{S}_0\simeq \widetilde{T}_0$. 
\end{lemma}

Therefore, we get a decomposition of $Spin(7,\mathbb{C})/G_2(\mathbb{C})$ as follows: 

\begin{proposition}
\label{prop:spin7-orbit}
$Spin(7,\mathbb{C})/G_2(\mathbb{C})=Spin(7)\cdot \widetilde{S}_0$. 
\end{proposition}

\subsection{Lie algebra of $\widetilde{A}_0$}
\label{subsec:a0-spin}
Let $\mathfrak{g}=\mathfrak{spin}(7,\mathbb{C})$, $\mathfrak{h}=\mathfrak{g}_2(\mathbb{C})$ 
and $\mathfrak{g}_u=\mathfrak{spin}(7)$ be the Lie algebras of 
$Spin(7,\mathbb{C})$, $G_2(\mathbb{C})$ and $Spin(7)$, respectively, and 
$\mathfrak{q}$ be the orthogonal complement of $\mathfrak{h}$ in $\mathfrak{g}$ 
with respect to the Killing form on $\mathfrak{g}$. 
In this subsection, 
the Lie algebra $\widetilde{\mathfrak{a}}_0$ of $\widetilde{A}_0$ is contained 
in $\sqrt{-1}\mathfrak{g}_u\cap \mathfrak{q}$. 

The Lie algebra $\widetilde{\mathfrak{a}}_0$ is given as follows. 
Let $\delta _{\theta }$ be a matrix given by 
\begin{align}
\label{eq:delta}
\delta _{\theta }:=\left( 
	\begin{array}{cc}
	0 & -\sqrt{-1}\theta \\
	\sqrt{-1}\theta & 0
	\end{array}
\right) . 
\end{align}
Then, $\widetilde{\mathfrak{a}}_0$ is given by 
\begin{align}
\label{eq:a0-tilde-algebra}
\widetilde{\mathfrak{a}}_0
	&=\{ \widetilde{\alpha }_{\theta }=\operatorname{diag}
	(\delta _{\theta },\delta _{(-\theta /3)},\delta _{(-\theta /3)},\delta _{(-\theta /3)}): 
	\theta \in \mathbb{R}\} . 
\end{align}
In particular, we have $\widetilde{A}_0=\exp \widetilde{\mathfrak{a}}_0$. 

We choose a Cartan involution $\theta $ of $\mathfrak{g}$ given by $\theta (X)=\overline{X}$ 
($X\in \mathfrak{g}$) (see Section \ref{subsec:cartan involution}). 
Then, we have 
$\theta (\widetilde{\alpha }_{\theta })=\overline{(\widetilde{\alpha }_{\theta })}
=-\widetilde{\alpha }_{\theta }$ 
for any $\widetilde{\alpha }_{\theta }\in \widetilde{\mathfrak{a}}_0$ 
because $\overline{\delta _{\theta }}=\delta _{(-\theta )}=-\delta _{\theta }$. 
This implies $\widetilde{\mathfrak{a}}_0\subset \sqrt{-1}\mathfrak{g}_u$. 

Next, we show: 

\begin{lemma}
\label{lem:a0-spin-q}
$\widetilde{\mathfrak{a}}_0\subset \mathfrak{q}$. 
\end{lemma}

\begin{proof}[Sketch of Proof]
As mentioned in Section \ref{subsec:a1}, 
the Lie algebra $\mathfrak{g}_2(\mathbb{C})$ is decomposed into 
the direct sum of $\mathfrak{sl}(3,\mathbb{C})$ 
and the orthogonal complement of $\mathfrak{sl}(3,\mathbb{C})$ in $\mathfrak{g}_2(\mathbb{C})$, 
denoted here by $\mathfrak{q}'$. 
Moreover, 
we take a $\mathbb{C}$-basis of $\mathfrak{sl}(3,\mathbb{C})$ as in (\ref{eq:sl3}) and 
that of $\mathfrak{q}'$ as in (\ref{eq:q}). 
It turns out that $\widetilde{\mathfrak{a}}_0$ is orthogonal to both 
$\mathfrak{sl}(3,\mathbb{C})$ and $\mathfrak{q}'$, 
and then to $\mathfrak{g}_2(\mathbb{C})$. 
Hence, we obtain $\widetilde{\mathfrak{a}}_0\subset \mathfrak{q}$. 
\end{proof}

Consequently, 
we have proved $\widetilde{\mathfrak{a}}_0\subset \sqrt{-1}\mathfrak{g}_u\cap \mathfrak{q}$. 

\subsection{Proof of Theorem \ref{thm:main} for $(Spin(7,\mathbb{C}),G_2(\mathbb{C}))$}
\label{subsec:proof-spin7}

In this subsection, we will prove Theorem \ref{thm:main} for 
$(Spin(7,\mathbb{C}),G_2(\mathbb{C}))$. 
Let ${G}_u:=Spin(7)$ be a maximal compact subgroup of ${G}_{\mathbb{C}}$. 

\begin{theorem}[Theorem \ref{thm:main} for Type R-1$'$]
\label{thm:cartan-spin7}
Let $(G_{\mathbb{C}},H_{\mathbb{C}})$ be the reductive spherical pair of Type R-1$'$, 
namely, $(Spin(7,\mathbb{C}),G_2(\mathbb{C}))$. 
Then, the one-dimensional abelian group $\widetilde{A}_0=\exp \widetilde{\mathfrak{a}}_0$ 
with $\widetilde{\mathfrak{a}}_0\subset \sqrt{-1}\mathfrak{g}_u\cap \mathfrak{q}$ 
given by (\ref{eq:a0-tilde-algebra}) satisfies 
\begin{align*}
{G}_{\mathbb{C}}={G}_u\widetilde{A}_0H_{\mathbb{C}}. 
\end{align*}
\end{theorem}

\begin{proof}
The proof of Theorem \ref{thm:cartan-spin7} is the same as the proof of Theorem \ref{thm:cartan-g2}. 
Then, we omit its proof. 
\end{proof}


\section{Cartan decomposition for $(SO(7,\mathbb{C}),G_2(\mathbb{C}))$}
\label{sec:so7g2}

In this section, 
we give a Cartan decomposition for the non-symmetric reductive spherical pair 
$(G_{\mathbb{C}},H_{\mathbb{C}})=(SO(7,\mathbb{C}),G_2(\mathbb{C}))$ (Type R-1). 
The key idea is to take the image of our Cartan decomposition for $(Spin(7,\mathbb{C}),G_2(\mathbb{C}))$ 
given by Theorem \ref{thm:cartan-spin7} through 
the double covering group homomorphism $\pi$ (see (\ref{eq:covering})). 

To carry out, 
we define a subgroup $A_0$ of $SO(7,\mathbb{C})=(SO(8,\mathbb{C}))_{e_0}$ by 
\begin{align}
\label{eq:a0}
A_0:=\{ a_{\theta }=\operatorname{diag}(I_2,d_{\theta },d_{\theta },d_{\theta }):\theta \in \mathbb{R}\} .
\end{align}
Here, an element $a_{\theta }$ has already appeared in the proof of Lemma \ref{lem:a0-spin} 
and $d_{\theta }$ is given by (\ref{eq:d}). 
Then, the Lie algebra $\mathfrak{a}_0$ of $A_0$ is of the form 
\begin{align}
\label{eq:a0-algebra}
\mathfrak{a}_0&=\{ \alpha _{\theta }
=\operatorname{diag}(O_2,\delta _{\theta },\delta _{\theta },\delta _{\theta }): 
	\theta \in \mathbb{R}\} 
\end{align}
where $\delta _{\theta}$ is given by (\ref{eq:delta}). 
In particular, we have $A_0=\exp \mathfrak{a}_0$. 

Let $\mathfrak{g}=\mathfrak{so}(7,\mathbb{C})$, $\mathfrak{h}=\mathfrak{g}_2(\mathbb{C})$ and 
$\mathfrak{g}_u=\mathfrak{so}(7)$ be the Lie algebras of $G_{\mathbb{C}}$, $H_{\mathbb{C}}$ and $G_u$, 
respectively 
and $\mathfrak{q}$ the orthogonal complement of $\mathfrak{h}$ in $\mathfrak{g}$. 
Clearly, $\mathfrak{a}_0$ is contained in $\sqrt{-1}\mathfrak{g}_u$. 
By the same argument as Lemma \ref{lem:a0-spin-q}, we have $\mathfrak{a}_0\subset \mathfrak{q}$. 
Thus, we obtain $\mathfrak{a}_0\subset \sqrt{-1}\mathfrak{g}_u\cap \mathfrak{q}$. 

Now, we return to the relation (\ref{eq:spin7}). 
This implies that $\pi$ induces the map $\pi :\widetilde{A}_0\to A_0$ given by 
$\pi (\widetilde{a}_{\theta })=a_{(2\theta /3)}$. 

\begin{lemma}
\label{lem:a0}
The image $\pi (\widetilde{A}_0)$ coincides with $A_0$. 
\end{lemma}

\begin{proof}
For any $a_{\theta '}\in A_0$, 
the element $\widetilde{a}_{(3\theta '/2)}\in \widetilde{A}_0$ satisfies 
$\pi (\widetilde{a}_{(3\theta '/2)})=a_{\theta '}$. 
Hence, we have proved $\pi (\widetilde{A}_0)=A_0$. 
\end{proof}

\begin{theorem}[Theorem \ref{thm:main} for Type R-1]
\label{thm:cartan-so7}
Let $(G_{\mathbb{C}},H_{\mathbb{C}})$ be the reductive spherical pair of Type R-1, 
namely, $(SO(7,\mathbb{C}),G_2(\mathbb{C}))$. 
Then, the one-dimensional abelian group $A_0=\exp \mathfrak{a}_0$ 
with $\mathfrak{a}_0\subset \sqrt{-1}\mathfrak{g}_u\cap \mathfrak{q}$ given by (\ref{eq:a0-algebra}) 
satisfies 
\begin{align*}
G_{\mathbb{C}}=G_uA_0H_{\mathbb{C}}. 
\end{align*}
\end{theorem}

\begin{proof}
We observe that $G_{\mathbb{C}}=SO(7,\mathbb{C})$ is realized 
as the image $\pi (\widetilde{G}_{\mathbb{C}})$ 
of $\widetilde{G}_{\mathbb{C}}=Spin(7,\mathbb{C})$. 
It follows from Theorem \ref{thm:cartan-spin7} that 
\begin{align}
\label{eq:image}
\pi (\widetilde{G}_{\mathbb{C}})=\pi (\widetilde{G}_u\widetilde{A}_0H_{\mathbb{C}})
=\pi (\widetilde{G}_u)\pi (\widetilde{A}_0)\pi (H_{\mathbb{C}}). 
\end{align}
Here, the image $\pi (\widetilde{G}_u)=\pi (Spin(7))$ coincides with $SO(7)$, 
$\pi (H_{\mathbb{C}})=\pi (G_2(\mathbb{C}))$ with $G_2(\mathbb{C})$, 
and by Lemma \ref{lem:a0} $\pi (\widetilde{A}_0)$ with $A_0$. 
Hence, (\ref{eq:image}) implies $G_{\mathbb{C}}=G_uA_0H_{\mathbb{C}}$. 
\end{proof}

The following theorem is an immediate consequence of Theorem \ref{thm:main} or 
Proposition \ref{prop:spin7-orbit}. 

\begin{proposition}
\label{prop:so7-orbit}
$SO(7,\mathbb{C})/G_2(\mathbb{C})=SO(7)\cdot (A_0G_2(\mathbb{C})/G_2(\mathbb{C}))$. 
\end{proposition}

For the sake of our application given in the next section, 
we will explain that Proposition \ref{prop:so7-orbit} follows from Proposition \ref{prop:spin7-orbit}. 

The double covering group homomorphism $\pi$ induces 
a double covering map 
\begin{align}
\label{eq:covering-space}
\widetilde{\pi }:Spin(7,\mathbb{C})/G_2(\mathbb{C})\to SO(7,\mathbb{C})/G_2(\mathbb{C}),\quad 
gG_2(\mathbb{C})\mapsto \pi (g)G_2(\mathbb{C}). 
\end{align}
In particular, $\widetilde{\pi}(Spin(7,\mathbb{C})/G_2(\mathbb{C}))$ 
coincides with $SO(7,\mathbb{C})/G_2(\mathbb{C})$. 
It follows from Proposition \ref{prop:spin7-orbit} that 
\begin{align*}
SO(7,\mathbb{C})/G_2(\mathbb{C})
&=\widetilde{\pi}(Spin(7,\mathbb{C})/G_2(\mathbb{C}))\\
&=\widetilde{\pi}(Spin(7)\cdot (\widetilde{A}_0G_2(\mathbb{C})/G_2(\mathbb{C})))\\
&=\pi (Spin(7))\cdot (\pi (\widetilde{A}_0)G_2(\mathbb{C})/G_2(\mathbb{C}))\\
&=SO(7)\cdot (A_0G_2(\mathbb{C})/G_2(\mathbb{C})). 
\end{align*}

Hence, we set 
\begin{align}
\label{eq:s0}
S_0:=A_0G_2(\mathbb{C})/G_2(\mathbb{C}). 
\end{align}
Then, the above argument shows: 

\begin{corollary}
\label{cor:s0}
$\pi (\widetilde{S}_0)=S_0$. 
\end{corollary}


\section{Application to visible actions on complex manifolds}
\label{sec:visible}

The motivation of the study for a Cartan decomposition for non-symmetric reductive spherical pairs 
is to investigate a classification problem on strongly visible actions 
on reductive complex homogeneous spaces. 
The notion of (strongly) visible actions has been introduced by T. Kobayashi 
for giving an unified explanation for multiplicity-freeness of representations (cf. \cite{mftheorem}). 
In this aspect, 
it plays a crucial role to find a real submanifold which meets every orbit. 
Once one can find a Cartan decomposition for a reductive spherical pair, 
one can also provide an explicit description of such a submanifold simultaneously. 
This section studies spherical homogeneous spaces of rank-one type 
from the viewpoint of (strongly) visible actions. 

Let us give a quick review on strongly visible actions. 
A holomorphic action of a Lie group $G$ on a connected complex manifold $D$ is called 
{\it strongly visible} if there exist a real submanifold $S$ in $D$ (called a {\it slice}) and 
an anti-holomorphic diffeomorphism $\sigma$ on $D$ satisfying the following conditions 
(see \cite{mftheorem}): 
\begin{gather}
\tag{V.1} \label{v1}
D=G\cdot S, \\
\tag{S.1} \label{s1}
\sigma |_{S}=\operatorname{id}_S, \\
\tag{S.2} \label{s2}
\sigma (x)\in G\cdot x\ (\forall x\in D). 
\end{gather}

We note that the slice $S$ is automatically totally real, 
namely, $J_x(T_xS)\cap T_xS=\{ 0\} $ for any $x\in S$ (see \cite[Remark 3.3.2]{mftheorem}). 
Here, $J$ stands for the complex structure of $D$. 

In Kobayashi's original definition \cite[Definition 3.3.1]{mftheorem} of strongly visible actions, 
it allows a complex manifold $D$ containing a non-empty $G$-invariant open set 
satisfying (\ref{v1})--(\ref{s2}). 
For this paper, 
we shall adopt the above definition for simplicity. 

Now, we prove: 

\begin{theorem}
\label{thm:visible-slice}
Let $(G_{\mathbb{C}},H_{\mathbb{C}})$ be a non-symmetric reductive spherical pair of rank-one type. 
Then, the $G_u$-action on $D=G_{\mathbb{C}}/H_{\mathbb{C}}$ is strongly visible. 
In particular, 
one can find a one-dimensional slice $S$ for the strongly visible action. 
\end{theorem}

In the following, 
we prove Theorem \ref{thm:visible-slice} for $(G_{\mathbb{C}},H_{\mathbb{C}})$ 
given in Table \ref{table:rank-one}. 
More precisely, 
we will verify three conditions (\ref{v1})--(\ref{s2}). 

\subsection{Verification of (\ref{v1})}
\label{subsec:v1}

We have already proved that there exists a one-dimensional real submanifold 
\begin{align}
\label{eq:slice}
S := AH_{\mathbb{C}}/H_{\mathbb{C}}. 
\end{align}
satisfying $D=G_u\cdot S$, which implies the condition (\ref{v1}). 
We list our choice of $S$ and the proposition showing $D=G_u\cdot S$ in Table \ref{table:v1}. 

\begin{table}[htbp]
\begin{align*}
\begin{array}{lcccc}
\hline 
\mbox{Type} & G_{\mathbb{C}} & H_{\mathbb{C}} & S & D=G_u\cdot S \\
\hline 
\mbox{R-1} & SO(7,\mathbb{C}) & G_2(\mathbb{C}) & S_0 & \mbox{Proposition \ref{prop:so7-orbit}}\\
\mbox{R-1}' & Spin(7,\mathbb{C}) & G_2(\mathbb{C}) & \widetilde{S}_0 
	& \mbox{Proposition \ref{prop:spin7-orbit}} \\
\mbox{R-2} & G_2(\mathbb{C}) & SL(3,\mathbb{C}) & S_1 & \mbox{Proposition \ref{prop:g2-orbit}}\\
\hline 
\end{array}
\end{align*}
\caption{Our choice of slice $S$ satisfying (\ref{s1})}
\label{table:v1}
\end{table}

\subsection{Verification of (\ref{s1})}
\label{subsec:s1}

In this subsection, 
we will verify the condition (\ref{s1}). 

First, let $I_{1,1}:=\operatorname{diag}(1,-1)$ and 
\begin{align*}
I_{+-}:&=\operatorname{diag}(I_{1,1},I_{1,1},I_{1,1},I_{1,1})
=\operatorname{diag}(1,-1,1,-1,1,-1,1,-1). 
\end{align*}
Since $(I_{+-}e_i)(I_{+-}e_j)=I_{+-}(e_ie_j)$ for $0\leq i,j\leq 7$, 
the element $I_{+-}$ lies in $G_2$. 
Here, 
we define an anti-holomorphic involution $\sigma _0$ on $SO(8,\mathbb{C})$ by 
\begin{align}
\label{eq:sigma0-so8}
\sigma _0(g)=I_{+-}\overline{g}I_{+-}\quad (g\in SO(8,\mathbb{C})). 
\end{align}

\begin{lemma}
\label{lem:stable}
The involution $\sigma _0$ stabilizes 
$Spin(7,\mathbb{C}),SO(7,\mathbb{C})$, $G_2(\mathbb{C})$ and $SL(3,\mathbb{C})$. 
\end{lemma}

\begin{proof}
As $I_{+-}\in G_2$, it is obvious that 
$\sigma _0$ stabilizes $Spin(7,\mathbb{C})$, $SO(7,\mathbb{C})$ and $G_2(\mathbb{C})$. 

For the proof that $SL(3,\mathbb{C})$ is $\sigma _0$-stable, 
it is necessary to verify the relation $\sigma _0(g)e_1=e_1$ for any $g\in SL(3,\mathbb{C})$. 
Let $g$ be an element of $SL(3,\mathbb{C})$. 
It is obvious that $\overline{g}\in SL(3,\mathbb{C})$. 
Thus, we obtain 
\begin{align*}
(\sigma _0(g))e_1=I_{+-}\overline{g}I_{+-}(e_1)=I_{+-}\overline{g}(-e_1)=I_{+-}(-e_1)=e_1. 
\end{align*}
Then, $\sigma _0$ stabilizes $SL(3,\mathbb{C})$. 
\end{proof}

By Lemma \ref{lem:stable}, 
the restrictions of $\sigma _0$ to 
$Spin(7,\mathbb{C})$, $SO(7,\mathbb{C})$ and $G_2(\mathbb{C})$ becomes involutions on 
$Spin(7,\mathbb{C})$, $SO(7,\mathbb{C})$ and $G_2(\mathbb{C})$, respectively, 
which we use the same letter $\sigma _0$ to denote. 

We choose a Cartan involution $\theta $ of $SO(8,\mathbb{C})$ as 
in (\ref{eq:theta}). 
Clearly, $\theta $ commutes with $\sigma _0$, 
from which $\sigma _0$ stabilizes the maximal compact subgroup 
$SO(8)=\{ g\in SO(8,\mathbb{C}):\theta (g)=g\}$ 
of $SO(8,\mathbb{C})$. 
By definition, $Spin(7),SO(7)$ and $G_2$ are also $\sigma _0$-stable. 

Let $(G_{\mathbb{C}},H_{\mathbb{C}})$ be a non-symmetric reductive spherical pair 
contained in Table \ref{table:rank-one}. 
The anti-holomorphic involution $\sigma _0$ on $G_{\mathbb{C}}$ induces 
an anti-holomorphic diffeomorphism $\sigma$ 
on the non-symmetric spherical homogeneous space $G_{\mathbb{C}}/H_{\mathbb{C}}$ 
as follows: 
\begin{align}
\label{eq:sigma}
\sigma (gH_{\mathbb{C}}):=\sigma _0(g)H_{\mathbb{C}}\quad (g\in G_{\mathbb{C}}). 
\end{align}

Now, let us show the condition (\ref{s1}) for our choice of $\sigma$ in (\ref{eq:sigma}). 
The submanifold $S$ in $G_{\mathbb{C}}/H_{\mathbb{C}}$ given in (\ref{eq:slice}) 
comes from 
the one-dimensional abelian subgroup in $G_{\mathbb{C}}$ 
given by $A_0$ (Type R-1), $\widetilde{A}_0$ (Type R-1$'$), $A_1$ (Type R-2). 
Then, it is necessary for the verification of (\ref{s1}) to show the following: 

\begin{lemma}
\label{lem:s1}
$\sigma _0|_{A}=\operatorname{id}_A$ for $A=A_0,\widetilde{A}_0,A_1$. 
\end{lemma}

\begin{proof}
First, 
let $a_{\theta }=\operatorname{diag}(I_2,d_{\theta },d_{\theta },d_{\theta })$ be an element of $A_0$. 
Then, the complex conjugation of $a_{\theta }$ equals $a_{(-\theta )}$. 
Thus, we compute 
\begin{align*}
\sigma _0(a_{\theta })
&=I_{+-}a_{(-\theta )}I_{+-}\\
&=\operatorname{diag}
	(I_{1,1}I_2I_{1,1},
		I_{1,1}d_{(-\theta )}I_{1,1},I_{1,1}d_{(-\theta )}I_{1,1},I_{1,1}d_{(-\theta )}I_{1,1})\\
&=\operatorname{diag}(I_2,d_{\theta },d_{\theta },d_{\theta })
=a_{\theta }. 
\end{align*}
Hence, $\sigma _0|_{A_0}=\operatorname{id}_{A_0}$ holds. 

Next, let $\widetilde{a}_{\theta }
=\operatorname{diag}(d_{\theta },d_{(-\theta /3)},d_{(-\theta /3)},d_{(-\theta /3)})$ 
be an element of $\widetilde{A}_0$. 
Then, we have 
\begin{align*}
\sigma _0(\widetilde{a}_{\theta })
&=\operatorname{diag}(I_{1,1}d_{(-\theta )}I_{1,1},I_{1,1}d_{(\theta /3)}I_{1,1}
	,I_{1,1}d_{(\theta /3)}I_{1,1},I_{1,1}d_{(\theta /3)}I_{1,1})\\
&=\operatorname{diag}(d_{\theta },d_{(-\theta /3)},d_{(-\theta /3)},d_{(-\theta /3)})
=\widetilde{a}_{\theta }. 
\end{align*}
This implies that $\sigma _0$ is the identity map on $\widetilde{A}_0$. 

Finally, 
let $t_{\theta }=\exp \tau _{\theta }
=\operatorname{diag}(\exp \delta _{(0,\theta )},\exp \delta _{(-\theta /2,-\theta /2)})$ 
be an element of $A_1$. 
We put $J_{+-}:=\operatorname{diag}(1,-1,1,-1)$. 
Then, we have $J_{+-}\delta _{(x,y)}J_{+-}=\delta _{(-x,-y)}$, 
Hence, we obtain 
\begin{align*}
\sigma _0(t_{\theta })
&=I_{+-}\operatorname{diag}(\exp \delta _{(0,-\theta )},\exp \delta _{(\theta /2,\theta /2)})I_{+-}\\
&=\operatorname{diag}(\exp (J_{+-}\delta _{(0,-\theta )}J_{+-}),
	\exp (J_{+-}\delta _{(\theta /2,\theta /2)}J_{+-}))\\
&=\operatorname{diag}(\exp \delta _{(0,\theta )},\exp \delta _{(-\theta /2,-\theta /2)})
=t_{\theta }. 
\end{align*}
Hence, $\sigma _0|_{A_1}=\operatorname{id}_{A_1}$. 

Therefore, Lemma \ref{lem:s1} has been proved. 
\end{proof}

Thanks to Lemma \ref{lem:s1}, the following equality holds 
for any $aH_{\mathbb{C}}\in S=AH_{\mathbb{C}}/H_{\mathbb{C}}$ ($a\in A$): 
\begin{align*}
\sigma (aH_{\mathbb{C}})=\sigma _0(a)H_{\mathbb{C}}=aH_{\mathbb{C}}. 
\end{align*}
Hence, we have verified $\sigma |_{S}=\operatorname{id}_S$, namely, (\ref{s1}). 

\subsection{Verification of (\ref{s2})}
\label{subsec:s2}

In this subsection, we shall see that the condition (\ref{s2}) follows from 
(\ref{v1}) and (\ref{s1}) and the involution $\sigma _0$ given by (\ref{eq:sigma0-so8}). 

Retain the setting as in Section \ref{subsec:s1}. 
Let $x$ be an element of the spherical homogeneous space $G_{\mathbb{C}}/H_{\mathbb{C}}$ 
of rank-one type. 
By the condition (\ref{v1}), 
one can find elements $k\in G_u$ and $a\in A$ such that $x=k\cdot aH_{\mathbb{C}}$. 
As $\sigma |_S=\operatorname{id}_S$ (the condition (\ref{s1})), we have 
\begin{align*}
\sigma (x)
=\sigma _0(k)\cdot \sigma (aH_{\mathbb{C}})
=\sigma _0(k)\cdot aH_{\mathbb{C}}
=(\sigma _0(k)k^{-1})\cdot x. 
\end{align*}
Here, $\sigma _0$ stabilizes $G_u$ (see Section \ref{subsec:s1}). 
Then, $\sigma _0(k)k^{-1}$ is an element of $G_u$. 
Hence, $(\sigma _0(k)k^{-1})\cdot x$ lies in the $G_u$-orbit through $x$, 
from which we have shown $\sigma (x)\in G_u\cdot x$. 
Hence, the condition (\ref{s2}) has been verified. 

\subsection{Proof of Theorem \ref{thm:visible-slice}}
\label{subsec:proof-visible}

For a non-symmetric reductive spherical pair $(G_{\mathbb{C}},H_{\mathbb{C}})$ of rank-one type, 
we have verified the condition (\ref{v1}) in Section \ref{subsec:v1}, 
(\ref{s1}) in Section \ref{subsec:s1} and (\ref{s2}) in Section \ref{subsec:s2}. 
Therefore, Theorem \ref{thm:visible-slice} has been proved. 

\subsection{Remark}
\label{subsec:remark}

We end this paper by the observation of $\sigma _0$ from 
the corresponding fixed point set of the Lie algebra as follows. 

Let $(G_{\mathbb{C}},H_{\mathbb{C}})$ be a non-symmetric spherical pair of rank-one type 
and $\mathfrak{g}$ the Lie algebra of a complex simple Lie group $G_{\mathbb{C}}$. 
We use the same letter $\sigma _0$ to denote the differential automorphism on $\mathfrak{g}$. 
We write $\mathfrak{g}^{\sigma _0}$ as the fixed point set of $\sigma _0$ in $\mathfrak{g}$. 

Our choice of $\sigma _0$ satisfies that 
$(\mathfrak{so}(8,\mathbb{C}))^{\sigma _0}$ is isomorphic to $\mathfrak{so}(4,4)$. 
Then, its real rank, denoted by 
$\operatorname{rank}_{\mathbb{R}}(\mathfrak{so}(8,\mathbb{C}))^{\sigma _0}$, 
equals four, which coincides with $\operatorname{rank}(\mathfrak{so}(8,\mathbb{C}))$. 
This means $(\mathfrak{so}(8,\mathbb{C}))^{\sigma _0}$ is a normal real form of 
$\mathfrak{so}(8,\mathbb{C})$. 
Further, we have 
\begin{gather*}
(\mathfrak{so}(7,\mathbb{C}))^{\sigma _0}\simeq \mathfrak{so}(7,\mathbb{C})\cap \mathfrak{so}(4,4)
\simeq \mathfrak{so}(3,4),\\
(\mathfrak{spin}(7,\mathbb{C}))^{\sigma _0}\simeq \mathfrak{spin}(7,\mathbb{C})\cap \mathfrak{so}(4,4)
\simeq \mathfrak{spin}(3,4),\\
(\mathfrak{g}_2(\mathbb{C}))^{\sigma _0}\simeq \mathfrak{g}_2(\mathbb{C})\cap \mathfrak{so}(4,4)
\simeq \mathfrak{g}_{2(2)}. 
\end{gather*}
It turns out that 
the Lie algebra $\mathfrak{g}^{\sigma _0}$ satisfies 

\begin{proposition}
\label{prop:rank}
$\operatorname{rank}_{\mathbb{R}}\mathfrak{g}^{\sigma _0}=\operatorname{rank}\mathfrak{g}$. 
\end{proposition}

We have found the same property as Proposition \ref{prop:rank} in the non-symmetric spherical 
homogeneous spaces of line bundle case. 
Namely, we have prove that if $G_{\mathbb{C}}/H_{\mathbb{C}}$ 
is a line bundle $G_{\mathbb{C}}/[K_{\mathbb{C}},K_{\mathbb{C}}]$ over the complexification 
$G_{\mathbb{C}}/K_{\mathbb{C}}$ of an irreducible Hermitian symmetric space $G/K$ of non-tube type, 
then the $G_u$-action on $G_{\mathbb{C}}/[K_{\mathbb{C}},K_{\mathbb{C}}]$ is strongly visible 
and one can take a slice $S$ and an anti-holomorphic diffeomorphism $\sigma$ satisfying 
(\ref{v1})--(\ref{s2}) and Proposition \ref{prop:rank} 
(see \cite[Lemma 2.2]{dedicata} and \cite{pja}). 
The key ingredient is to find a Cartan decomposition 
for $(G_{\mathbb{C}},[K_{\mathbb{C}},K_{\mathbb{C}}])$ explicitly. 

We can show that for any reductive spherical pair $(G_{\mathbb{C}},H_{\mathbb{C}})$ 
we have a Cartan decomposition by giving an explicit description of the abelian part 
and that the $G_u$-action on the spherical homogeneous space $G_{\mathbb{C}}/H_{\mathbb{C}}$ 
is strongly visible 
with a slice coming from a Cartan decomposition for $(G_{\mathbb{C}},H_{\mathbb{C}})$ 
and an anti-holomorphic diffeomorphism coming from an involution on $G_{\mathbb{C}} $ 
satisfying Proposition \ref{prop:rank}. 
In fact, 
we have shown the strong visibility in some cases, 
see \cite{dedicata, jms, apam, pja}, 
in particular, we have provided a slice explicitly 
(note that our choice of slice in \cite{jms,apam} is not abelian). 
The others will be explained in the forthcoming papers 
which contain how to find an explicit description of the abelian part for a Cartan decomposition 
(cf. \cite{cayley}). 



\end{document}